\documentclass[10pt]{article}
\usepackage{amsthm, amstext}
\usepackage{amsmath}
\usepackage{amsmath, amsthm, amssymb}

\usepackage{enumitem}
\usepackage{graphicx}
\usepackage{amsfonts}
\usepackage{latexsym, xypic}
\usepackage[all]{xy}
\xyoption{all}
\usepackage[colorlinks,linkcolor=blue,citecolor=blue]{hyperref}
\theoremstyle{plain}
\newtheorem{theorem}{Theorem}[section]
\newtheorem{lemma}[theorem]{Lemma}
\newtheorem{proposition}[theorem]{Proposition}

\theoremstyle{definition}
\newtheorem{definition}[theorem]{Definition}

\newtheorem{example}{\sc Example}
\theoremstyle{remark}
\newtheorem{remark}{\sc Remark}
\theoremstyle{case}
\newtheorem{case}{Case}

\def\pi{positive implicative }

\date{}
\begin{document}
	\title{\bf {Maximal  and Frattini $L$-Subgroups of an $L$-Group}}
	\author{\textbf{Iffat Jahan$^1$ and Ananya Manas $^2$} \\\\ 
			$^{1}$
			Department of Mathematics, Ramjas College\\
			University of Delhi, Delhi, India \\
			ij.umar@yahoo.com \\\\
			$^{2}$Department of Mathematics, \\
			University of Delhi, Delhi, India \\
			anayamanas@gmail.com 
			\\    }
	\date{}
	\maketitle
		\begin{abstract}
		 In this paper, the concept of a maximal $L$-subgroup of an $L$-group has been defined in the spirit of classical group theory. Then, a level subset characterization has been established for the same. Then, this notion of maximal $L$-subgroups  has been used to define Frattini $L$-subgroups.  Further, the concept of non-generators of an $L$-group has been developed and  its relation with the Frattini $L$-subgroup of an $L$-group has been established like their classical counterparts. Moreover, several properties pertaining to the concepts of maximal $L$-subgroups and Frattini $L$-subgroup have also been investigated. These two notions have been illustrated through several examples.	
		\end{abstract}
		
		\section{Introduction}
		After the pioneer work of Zadeh \cite{zadeh_fuzzy}, the applications of fuzzy set theory and fuzzy logic in a number of diverse fields are well known by now. Also, in Mathematics, group theory has always been in the forefront and has found its practical applications in various fields of science and technology. As a result of natural progression, in 1971, A. Rosenfeld \cite{rosenfeld_fuzzy} applied the notion of fuzzy sets to subgroupoids and groups which instigated the studies of fuzzy algebraic structures. On the other hand, lattice theory has been effectively applied to various branches of science and technology.  In information sciences various branches such as computational intelligence, neural networks, pattern recognition, mathematical morphology can be unified with an application of lattice theory.  In fact, diverse concepts can be studied under the purview of lattice theory. Hopefully, such an amalgamation of lattices and fuzzy subgroups will open doors for new applications while revealing deeper structure of fuzzy subgroups. A study of fuzzy algebraic structures and lattices came into the existence, in the year 1981, when Liu \cite{liu_op} introduced the lattice valued fuzzy subgroups. In \cite{mordeson_comm, malik_pri}, Mordeson and Malik have developed $L$-ring theory (lattice valued fuzzy ring theory) in a systematic way like its classical counterpart. It is worthwile to mention here that in \cite{mordeson_comm} the parent structure is an ordinary ring rather than a latice valued fuzzy ring ($L$-ring). This setting has its own limitations and does not even allow the formulation of various concepts of classical algebra in fuzzy or $L$-(lattice valued fuzzy) setting. This drawback can be removed easily if the parent structure considered in the definition of an $L$(fuzzy)-algebraic concept is an $L$(fuzzy)-algebraic structure rather than an ordinary algebraic structure \cite{ajmal_char}. In fact, very few researchers such as Martinez \cite{martinez_fuzzy} have studied the properties of a $L$-subring of an $L$-ring.In \cite{prajapati_max1, prajapati_max2}, Ajmal and Prajapati have introduced the notion of maximal $L$-ideals of an $L$-ring with an essence similar to classical ring theory. In fact, such a definition of maximal $L$-ideal   could be formulated as the parent structure considered in this definition is an $L$-ring rather than an ordinary ring.
		However, in the studies of fuzzy groups such an effort is lacking.  Recently, a systematic study of $L$-subgroups (lattice valued fuzzy subgroups) of an $L$-group has been carried out in a series of papers \cite{ajmal_char, ajmal_nc, ajmal_nil, ajmal_nor, ajmal_sol} wherein a number of concepts of  classical group theory have been extended to $L$-setting  specially keeping in view their compatibility. The present paper is an endeavour to develop and study the maximal $L$-subgroup of an $L$-group along with its application to the notion of Frattini subgroups.

		In Section 3, the concept of maximal subgroups has been extended to the $L$-setting. The maximal $L$-subgroup of an $L$-group $\mu$ is defined to be a proper $L$-subgroup that is not properly contained in any other $L$-subgroup of $\mu$. Then, a level subset characterization of maximal $L$-subgroup of an $L$-group $\mu$ has been provided, provided that $\eta$ and $\mu$ are jointly supstar and both $\eta$ and $\mu$ have the same tips.  This characterization has been effectively applied to develop the notion of Frattini $L$-subgroup of $\mu$ in Section 4. A sufficient condition for an $L$-subgroup of $\mu$ to be a maximal $L$-subgroup has also been established.  
		
		Section 4 explores the concept of Frattini subgroup in $L$-setting. The Frattini subgroup of a group is a significant concept in classical group theory. It is defined as the intersection of all maximal subgroups of a group and in case a group has no maximal subgroups, the Frattini subgroup is defined to be the group itself. Thus it is comparable to the concept of Jacobson radical in ring theory. Another important property of Frattini subgroup is that it coincides with the subgroup of non-generators of a group. Therefore it is considered as the subgroup of "small elements".  Section 4 starts with the definition of the Frattini $L$-subgroup $\Phi(\mu)$ of an $L$-group $\mu$. It is defined as the intersection of all maximal $L$-subgroups of $\mu$. In case $\mu$ has no maximal $L$-subgroups, $\Phi(\mu)$ is defined to be $\mu$ like its classical counterpart. Next, the notion of non-generators of an $L$-group has been introduced. Then, in upper well ordered lattices, it has been established that the Frattini $L$-subgroup of $\mu$ is same as the $L$-subgroup generated by the union of non-generators of $\mu$.  Further, it has been shown that if $\mu$ is a normal $L$-subgroup of a group $G$, then the Frattini $L$-subgroup $\Phi(\mu)$ is a normal $L$-subgroup of $\mu$. The paper ends with an investigation of the images and pre-images of the Frattini $L$-subgroup under group homomorphisms.    
		    
\section{Preliminaries}

Throughout this paper, the system $\langle L, \leq,\vee,\wedge\rangle $  denotes a completely distributive lattice where $\leq $  denotes the partial ordering of $L$, the join (sup) and the meet (inf) of the elements of $L$  are denoted by
$‘\vee’$  and $’\wedge’$, respectively. Also, we write 1 and 0 for the maximal and the minimal elements of $L$, respectively. Moreover, our work is carried out by using the definition of $L$-subsets as formulated by Goguen \cite{goguen_sets}. The definition of a completely distributive lattice is well known in the literature and can be found in any standard text on the subject \cite{gratzer_lattices}. \

\vspace{.3 cm}

\noindent Let  $\{J_i:i\in I\}$ be any family of subsets of a complete lattice $L$ and  $F$ denotes the set of choice functions for  $J_i$,
that is,  functions $f:I\to\prod\limits_{i\in I}J_i$ such that $f(i)\in J_i$ for each $i\in I$.\ Then, we say that $L$ is a completely distributive lattice, if
\[
\begin{array}{l}
\bigwedge\left\{\bigvee_{i\in I}J_i\right\} =\bigvee_{f\in F}\left\{\bigwedge_{i\in I}f(i)\right\}.
\end{array}
 \]

\noindent The above law is known as the complete distributive law. Moreover, a lattice $L$ is said to be infinitely meet distributive if for every subset $\{b_\beta:\beta\in B\}$
 of $L$, we have
 \begin{center}
 $a\bigwedge\lbrace\bigvee\limits_{\beta\in B}b_\beta\rbrace=\bigvee\limits_{\beta\in B}\lbrace a\bigwedge b_\beta\rbrace$,
 \end{center}
\noindent provided $L$ is join complete. The above law is known as the infinitely meet distributive law. The definition of infinitely join distributive lattice is dual to the above definition, that is, a lattice $L$ is said to be infinitely join distributive if for every subset $ \{b_\beta : \beta \in B \}$  of $L$, we have
 \begin{center}
 $a\bigvee\lbrace\bigwedge\limits_{\beta\in B}b_\beta\rbrace=\bigwedge\limits_{\beta\in B}\lbrace a \vee b_ \beta\rbrace,$
 \end{center}
provided $L$ is meet complete. The above law is known as the infinitely join distributive law.
Clearly, both these laws follow from the definition of a completely distributive lattice. Here we also mention that the dual of completely distributive law is valid in a completely distributive lattice whereas the infinitely meet and join distributive laws are independent from each other. Next, we recall the following from \cite{ajmal_sup,  ajmal_gen, ajmal_nil, goguen_sets, mordeson_comm, wu_normal}:\\

An $L$-subset of a non-empty set $X$ is a function from $X$  into $L$.  The set of  $L$-subsets of $X$ is called  the $L$-power set of $X$ and is denoted by $L^X$.  For  $\mu \in L^X, $  the set $ \lbrace\mu(x) \mid x \in X \rbrace$  is called the image of $\mu$  and is denoted by  Im $\mu $ and the tip of $ \mu $  is defined as $\bigvee \limits_{x \in X}\mu(x). $ Moreover, the tail of $\mu$ is defined as $\bigwedge \limits_{x \in X}\mu(x). $ We say that an $L$-subset $\mu$  of $X$ is contained in an $L$-subset $\eta$  of $X$ if  $\mu(x)\leq \eta (x)$   for all $x \in X$. This is denoted by $\mu \subseteq \eta $.  For a family $\lbrace\mu_{i} \mid i \in I \rbrace$  of $L$-subsets in  $X$, where $I$  is a non-empty index set, the union $\bigcup\limits_{i \in I} \mu_{i} $    and the intersection  $\bigcap\limits_{i \in I} \mu_{i} $ of  $\lbrace\mu_{i} \mid i \in I \rbrace$ are, respectively, defined by:
   \begin{center}
   $\bigcup\limits_{i \in I} \mu_{i}(x)= \bigvee\limits_{i \in I} \mu(x)  $ and $\bigcap\limits_{i \in I} \mu_{i} (x)= \bigwedge\limits_{i \in I} \mu(x), $
\end{center}
for each  $x \in X $. If  $\mu \in L^X $  and  $a \in L $,  then the notion of  level  subset $\mu_{a}$   of $\mu$  is defined as:
 \begin{center}
 $\mu_{a}= \lbrace x \in X \mid \mu (x) \geq a\rbrace.$
 \end{center}
 For $\mu, \nu \in L^{X} $, it can be verified easily that if $\mu\subseteq \nu$, then $\mu_{a} \subseteq \nu_{a} $ for each $a\in L $. Also, the following result is well known in the literature:\\
 \begin{proposition}
 	\label{int_lev}
 	Let $\{\eta_{i}\}_{i \in I}$ be a family of $L$-subsets of $X$. Then, 
 	$$\left\{\bigcap_{i\in I}\eta_{i}\right\}_{a}=\bigcap_{i\in I}\{\eta_{i}\}_{a}.$$
 	
 \end{proposition}
 
If $a\in L$ and $x \in X$, then we define $a_{x} \in L^{X} $ as follows:
\[
a_{x} ( y ) =
\begin{cases}
a &\text{if} \ y = x,\\
0 &\text{if} \ y\ne x.
\end{cases}
\]
$a_{x} $ is referred to as an $L$-point or $L$-singleton. We say that $a_{x} $ is an $L$-point of $\mu$ if and only if
$\mu( x )\ge a$ and we write $a_{x} \in \mu$. 
The set product $\mu \circ \eta $   of $\mu, \eta \in L^S$, where $S$ is a groupoid, is an $L$-subset of $S$ defined by
\begin{center}
$\mu \circ \eta (x) = \bigvee \limits_{x=yz}\lbrace\mu (y) \wedge \eta (z) \rbrace.$
\end{center}

\noindent Here we point out that if $x$ cannot be factored as  $x=yz$  in $S$, then  $\mu \circ \eta (x)$, being  the least upper bound of the empty set, is zero. It can be verified easily that the set product is associative in  $L^S$  if $S$ is a semigroup.\\

Let $f$ be a mapping from a set $X$ to a set $Y$. If $\mu \in L ^{X}$ and $\nu \in L^{Y}$, then the image $f(\mu )$
of $\mu $ under $f$ and the preimage $f^{-1} (\nu )$ of $\nu $ under $f$ are $L$-subsets of $Y$ and $X$ respectively, defined by
\[
f(\mu )(y)=\bigvee\limits_{x\in f^{-1} (y)} \{\mu (x)\}\quad\text{and}\quad
f^{-1} (\nu )(x)=\nu (f(x)).
\]
Again, recall that  if $f^{-1} (y)=\phi $,
then $f(\mu )(y),$ being the least upper bound of the empty set, is zero.

\begin{proposition}
	\label{hom}
	Let $f : X \rightarrow Y$ be a mapping.
	\begin{enumerate}
		\item[{(i)}] Let $\{ \mu_i \}_{i \in I}$ be a family of $L$-subsets of $X$. Then, $f(\mathop{\cup}\limits_{i \in I} \mu_i) = \mathop{\cup}\limits_{i \in I} f(\mu_i)$ and $f(\mathop{\cap}\limits_{i \in I} \mu_i) \subseteq \mathop{\cap}\limits_{i \in I}f(\mu_i)$.
		\item[{(ii)}] Let $\mu \in L^X$. Then, $f^{-1}(f(\mu)) \supseteq \mu$. The equality holds if $f$ is injective.
		\item[{(iii)}] Let  $\nu \in L^Y$. Then, $f(f^{-1}(\nu)) \subseteq \nu$. The equality holds if $f$ is surjective.
		\item[{(iv)}] Let $\mu \in L^X$ and $\nu \in L^Y$. Then, $f(\mu) \subseteq \nu$ if and only if $\mu \subseteq f^{-1}(\nu)$. Moreover, if $f$ is injective, then $f^{-1}(\nu) \subseteq \mu$ if and only if $\nu \subseteq f(\mu)$.
	\end{enumerate}
\end{proposition}

Throughout this paper $G$ denotes an ordinary group with the identity element `$e$' and $I$ denotes a non-empty indexing set. Also, $1_A$ denotes the characteristic function of a non-empty set $A$.

\begin{definition}
	Let $\mu \in L ^G $. Then, $\mu $ is called an $L$-subgroup of $G$ if for each $x, y\in G$,
	\begin{enumerate}
		\item[({i})] $\mu (xy)\ge \mu (x)\wedge \mu (y)$,
		\item[({ii})] $\mu (x^{-1} )=\mu (x)$.
	\end{enumerate}
	The set of $L$-subgroups of $G$ is denoted by $L(G)$. Clearly, the tip of an $L$-subgroup
	is attained at the identity element of $G$.
\end{definition}

\begin{theorem}
	\label{lev_gp}
	Let $\mu \in L ^G $. Then, $\mu $ is an $L$-subgroup of $G$ if and only if each non-empty level subset $\mu_{a} $ is a subgroup of $G$.
\end{theorem}

It is well known in literature that the intersection of an arbitrary family of $L$-subgroups of a group is an $L$-subgroup of the given group.

\begin{definition}
	Let $\mu \in L ^G $. Then, the $L$-subgroup of $G$ generated by $\mu $ is defined as the smallest $L$-subgroup of $G$
	which contains $\mu $. It is denoted by $\langle \mu \rangle $, that is,
	\[
	\langle \mu \rangle =\cap\{\mu _{{i}} \in L(G) \mid \mu \subseteq \mu _{i}\}.
	\]
\end{definition}

\begin{definition}
	Let $\mu\in L(G)$. Then, $\mu $ is called a normal $L$-subgroup of $G$ if for all $x, y \in  G$, $\mu ( xy ) = \mu ( yx )$.
\end{definition}

\noindent The set of normal $L$-subgroups $G$ is denoted by $NL(G)$. 

\begin{theorem}
	\label{lev_norgp}
	Let $\mu \in L{(G)}$. Then, $\mu \in NL(G)$ \text{if and only if each non-empty level subset~} $\mu_a$ \text {~is a normal subgroup of~} $G$.	
\end{theorem}

\begin{definition}
	Let $\eta, \mu\in L^{{G}}$ such that $\eta\subseteq\mu$. Then, $\eta$ is said to be an $L$-subset of $\mu$. The set of all $L$-subsets of $\mu$ is denoted by $L^{\mu}.$
	Moreover, if $\eta,\mu\in L(G)$ such that  $\eta\subseteq \mu$, then $\eta$ is said to be an $L$-subgroup of $\mu$. The set of all $L$-subgroups of $\mu$ is denoted by $L(\mu)$.
\end{definition}

\begin{definition} 
	Let $\eta\in L(\mu)$ such that $\eta$ is non-constant and $\eta\ne\mu$. Then, $\eta$ is said to be a proper $L$-subgroup of $\mu$.
\end{definition}

\noindent Clearly, $\eta$ is a proper $L$-subgroup of $\mu$ if and only if $\eta$ has distinct tip and tail and $\eta\ne\mu$.
\vspace{.1cm}

\begin{theorem}
	\label{lev_sgp}
	Let $\eta \in L^\mu$. Then, $\eta\in L(\mu)$ if and only if each non-empty level subset $\eta_a$  is a subgroup of $\mu_a$.
\end{theorem}

\begin{theorem}
	\label{hom_gp}
	Let $f : G \rightarrow H$ be a group homomorphism. Let $\mu \in L(G)$ and $\nu \in L(H)$. Then, $f(\mu) \in L(H)$ and $f^{-1}(\nu) \in L(G)$.
\end{theorem}

We shall have an $L$-group as our parent group which will be denoted by $\mu$ throughout our work. We recall  the definition of a normal $L$-subgroup of an $L$-group and some results  which are used in the development of this paper.
\begin{definition}\label{2.5}
Let $\eta \in L(\mu)$. Then, we say that  $\eta$  is a normal $L$-subgroup of $\mu$   if  
\begin{center}
$\eta(yxy^{-1}) \geq \eta(x)\wedge \mu(y)$ for  all  $x,y \in G.$
\end{center}
\end{definition}

\noindent The set of normal $L$-subgroups of $\mu$  is denoted by $NL(\mu)$. If $\eta \in NL(\mu)$, then we write\vspace{.2cm} $ \eta \triangleleft \mu$. 

Here we mention that the arbitrary intersection of a family of normal $L$-subgroups of an $L$-group $\mu$ is again a normal $L$-subgroup of $\mu$.\\\\
\text{REMARK.}
	It is important to note that $\mu$ is a normal $L$-subgroup of $G$ if and only if $\mu\in NL(1_G)$

\vspace{.2cm}

\begin{theorem}
	\label{lev_norsgp}
	Let $\eta \in L(\mu)$. Then, $\eta\in NL(\mu) \text{~if and only if each non-empty level subset~} \eta_a \\ \text {is a normal subgroup of~} \mu_a$.
\end{theorem}

\begin{definition}
	Let $\mu \in L^X$. Then, $\mu$ is to have sup-propery if for each $A \subseteq X$, there exists $a_0 \in A$ such that $ \mathop \vee \limits_{a \in A}  {\mu(a) } = \mu(a_0)$. 
\end{definition}

Lastly, recall the following from \cite{ajmal_gen, ajmal_sol}:

\begin{theorem}
	\label{gen}
	Let $\eta\in L^{^{\mu}}.$ Let $a_{0}=\mathop {\vee}\limits_{x\in G}{\left\{\eta\left(x\right)\right\}}$ and define an $L$-subset $\hat{\eta}$ of $G$ by
	\begin{center}
		$\hat{\eta}\left(x\right)=\mathop{\vee}\limits_{a \leq a_{0}}{\left\{a \mid x\in\left\langle \eta_{a}\right\rangle\right\}}$.
	\end{center}
	
	\noindent Then, $\hat{\eta}\in L(\mu)$ and  $\hat{\eta} =\left\langle \eta \right\rangle$.
\end{theorem}

\begin{theorem}
	\label{gen_sup}
	Let $\eta \in L^{\mu}$ and possesses the sup-property. If $a_0 = \mathop{\vee}\limits_{x \in G}\{\eta(x)\}$, then for all $b \leq a_0$, $\langle \eta_b \rangle = \langle \eta \rangle_b$.
\end{theorem}

\begin{theorem}
	\label{gen_hom1}
	Let $f : G \rightarrow H$ be a group homomorphism and let $\mu \in L(G)$. Then, for all $\eta \in L^{\mu}$, $\langle f(\eta) \rangle = f(\langle \eta \rangle).$
\end{theorem}

\begin{theorem}
	\label{gen_hom2}
	Let $f : G \rightarrow H$ be a group homomorphism and let $\nu \in L(H)$. Then, for all $\theta \in L^{\nu}$, $\langle f^{-1}(\theta) \rangle = f^{-1}(\langle \theta \rangle).$
\end{theorem}

\section{Maximal $L$-subgroups of an $L$-group}
 Prajapati and  Ajmal \cite{prajapati_max1, prajapati_max2} have developed the notion of maximal $L$-ideals of an $L$-ring in the spirit similar to that of maximal ideals in classical ring theory. However, in the studies of $L$-subgroups such an effort is lacking. This provided us sufficient motivation for the development of  maximal $L$-subgroups of an $L$-group. Here we formulate the maximal $L$-subgroup of an $L$-group.
\begin{definition}
	Let $\mu \in L(G)$. A proper $L$-subgroup $\eta$ of $\mu$ is said to be a maximal $L$-subgroup of $\mu$ if whenever $\eta \subseteq \theta \subseteq \mu$ for some $\theta \in L(\mu)$, then either $\theta = \eta$ or $\theta = \mu$.
\end{definition}

The following result describes a relation of the tip of a maximal $L$-subgroup of an $L$-group $\mu$ with that of the tip of parent $L$-group $\mu$. 

\begin{proposition}
	Let $\eta \in L(\mu)$ be a maximal $L$-subgroup. Then, $\eta(e) = \mu(e)$ or $\eta(e)$ is a cover of $\mu(e)$. 
\end{proposition}
\begin{proof}
	Let $\eta(e) \neq \mu(e)$ and suppose there exists $a_{1} \in L$ such that $\eta(e) < a_1 < \mu(e)$. Define $\theta : G \rightarrow L$ as follows:
	\[ \theta (x) = \begin{cases}
	a_1 & \text{if } x = e, \\
	\eta(x) & \text{if } x \neq e.  
	\end{cases} \]
	Now, the following is easy to verify: 
	\begin{center}
		$\theta_a =	\{ e \}$ if $a = a_1$ and 
		$\theta_a =\eta_a $ if  $a \neq a_1$.
	\end{center}
	\noindent Thus each non-empty level subset $\theta_a$ is a subgroup of $\mu_a$. Hence by Theorem \ref{lev_sgp}, $\theta \in L(\mu)$. Clearly,  $\eta \subsetneq \theta$. Also, $\theta(e) < \mu(e)$ and hence $ \theta \subsetneq \mu$. This contradicts the maximality of $\eta$ in $\mu$. Hence  $\mu(e)$ must be a cover of $\eta(e)$.
\end{proof}

The notion of maximal $L$-subgroup of an $L$-group has been illustrated in the following example:
 
\begin{example}
	Let G be the quaternian group $Q_8$ given by :
		\begin{center}
			$Q_8=\{ \pm 1,\pm i, \pm j, \pm k \},$
		\end{center}
		where $\ i^2=j^2=k^2=-1, ij=k, jk=i, kj=i$. Let $C=\{1,-1\}$  be the center of  $G$ and $H=\{\pm 1, \pm i\}$. Let the evaluation lattice $L$ be the chain given by :
		$$
		L:  0 < a < b < c < 1.
		$$
		Define $L$-subsets $\mu$ and $\eta$ of $G$ as follows:
		
		\[ \mu (x)=\begin{cases} 
		1 & \text{if }  x \in C, \\
		b & \text{if }  x \in H \setminus C,\\
		a & \text{if }  x \in G \setminus H.
		\end{cases}	\]
		
		and
		
		\[ \eta (x)= \begin{cases}
		1 & \text{if } x = 1,\\
		c & \text{if } x \in C \setminus \{1\}, \\
		b & \text{if } x\in H \setminus C,\\
		a & \text{if } x\in G \setminus H.
		\end{cases} \]
				
		\noindent Since the non-empty level subsets of $\eta$ and $\mu$ are subgroups of $G$,  $\eta$ and $\mu$ are $L$-subgroups of $G$. As $\eta \subseteq \mu$, $\eta$ is an $L$-subgroup of $\mu.$ We show that $\eta$ is a maximal $L$-subgroup of $\mu$. Suppose there exists $\theta \in L(\mu)$ such that $\eta \subsetneq \theta \subseteq \mu$. Then, since $\eta(x) = \mu(x)$ for all $x \neq -1$, we must have $\theta(-1) > c = \eta(-1)$ and $\theta(x) = \eta(x)$ for all $x \neq -1$. But then, we must have $\theta(-1) = 1 =\mu(e)$ and hence $\theta =\mu$. Thus there does not exist any $\theta \in L(\mu)$ such that $\eta \subsetneq \theta \subsetneq \mu$. We conclude that $\eta$ is a maximal $L$-subgroup of $\mu$.
\end{example}

\begin{remark}
In order to  study the level subsets of maximal $L$-subgroups of an $L$-group, we recall the notion of jointly supstar $L$-subsets from \cite{ajmal_nil}. It is worthwile to mention here that this notion is a generalization of the noion of sup-property and lends itself easily for applications
\end{remark}

\begin{proposition} \label{supchar}
	Let $\eta \in L^{\mu}$. Then, $\eta$ possesses the sup-property if and only if every subset of $\text{Im}~\eta$ is closed under arbitrary supremums. 
\end{proposition}

\begin{definition}
	A non-empty subset $X$ of a lattice $L$ is said to be a supstar subset of $L$ if every non-empty subset $A$ of $X$ contains its supremum.
\end{definition}

\begin{definition}
	Let $\{ \eta_i \}_{i \in I}$ be a family of $L$-subsets of $\mu$. Then, $\{\eta_i\}_{i \in I}$ is said to be a supstar family if $\bigcup\limits_{i \in I}^{} \text{Im}~\eta_i$ is a supstar subset of $L$. As a particular case, we say that two $L$-subsets $\eta$ and $\theta$ are jointly supstar if $\text{Im}~\eta \cup \text{Im}~\theta$ is a supstar subset of $L$.
\end{definition}

In the following theorem, we describe level subsets of maximal $L$-subgroups of an $L$-group:

\begin{theorem}
	\label{max_sup1}
	Let $\eta \in L(\mu)$ be such that $\mu$ and $\eta$ are jointly supstar. Let $\eta$ be a maximal $L$-subgroup of $\mu$. Then, there exists exactly one $a_0 \in \text{Im}~\mu$ such that $\eta_{a_0} \subsetneq \mu_{a_0}$ and for all other $a \in \text{Im}~\mu \cup \text{Im}~\eta$, $\eta_a = \mu_a$.
\end{theorem}
\begin{proof}
	Since $\eta$ is a maximal $L$-subgroup of $\mu$,  $\eta \subsetneq \mu$. Hence there exists $x \in G$ such that $\eta(x) < \mu(x)$. Let $a_0 = \mu(x)$. Then, $a_0 \in \text{Im}~\mu$ and $x \in \mu_{a_0} \setminus \eta_{a_0}$. Hence $\eta_{a_0} \subsetneq \mu_{a_0}$.
	Now, let $a_1 \in \text{Im}~\mu\cup \text{Im}~\eta$ such that $a_1 \neq a_0$ and $\eta_{a_1} \subsetneq \mu_{a_1}$. Since $\{ a_0, a_1 \} \subseteq \text{Im}~\mu \cup \text{Im}~\eta$ and by the hypothesis $\eta$ and $\mu$ are jointly supstar, it follows that
	\[ a_0 \vee a_1 = a_0 \text{ or } a_1. \]
	Without loss of generality, we may assume that 
	\[ a_0 \vee a_1 = a_1, \]
	that is, $a_0 < a_1$.
	Define $\theta : G \rightarrow L$ as follows :
	\[ \theta(x) = \{\eta(x) \vee a_0\} \wedge \mu(x)  \text{~~~for all~} x \in G.\]
	 Firstly, we show that $\theta \in L(G)$. Let $x$, $y \in G$. Then,
	\begin{equation*}
	\begin{split}
	\theta(xy) & = \{\eta(xy) \vee a_0\} \wedge \mu(xy) \\
	& \geq \{(\eta(x) \wedge \eta(y)) \vee a_0 \} \wedge \{ \mu(x) \wedge \mu(y) \} \\
	& = \{ (\eta(x) \vee a_0) \wedge (\eta(y) \vee a_0) \} \wedge \{ \mu(x) \wedge \mu(y) \} \\
	& = \{ (\eta(x) \vee a_0) \wedge \mu(x) \} \wedge \{ (\eta(y) \vee a_0) \wedge \mu(y) \} \\
	& = \theta(x) \wedge \theta(y).	
	\end{split}
	\end{equation*}
	Also,
	\[ \theta(x^{-1}) = \{\eta(x^{-1}) \vee a_0\} \wedge \mu(x^{-1}) =\{\eta(x) \vee a_0\} \wedge \mu(x) = \theta(x).  \]
	Hence $\theta \in L(G)$.
	Now, 
	\[ \eta(x) \leq \eta(x) \vee a_0 
\text{~	and ~}
	 \eta(x) \leq \mu(x) ~~~~~\text{~for all~} x \in G. \]
	Therefore 
	\[ \eta(x) \leq (\eta(x) \vee a_0) \wedge \mu(x) = \theta(x) \leq \mu(x) \text{~~~for all~} x \in G, \]
	that is, $\eta \subseteq \theta \subseteq \mu$. Now, since $\eta_{a_0} \subsetneq \mu_{a_0}$, there exists $x_0 \in \mu_{a_0}$ such that $x_0 \notin \eta_{a_0}$. As $\eta$ and $\mu$ are jointly supstar and $\{ \eta(x_0), a_0 \} \subseteq \text{Im}~\mu \cup \text{Im}~\eta$, it follows that 
	\[ \eta(x_0) \vee a_0 = \eta(x_0) \text{ or } a_0. \]
	However, $\eta(x_0) \vee a_0 \neq \eta(x_0)$. For, if $\eta(x_0) \geq a_0$, then $x_0 \in \eta_{a_0}$, which is contrary to our assumption that $x_0 \notin \eta_{a_0}$. Therefore $\eta(x_0) \vee a_0 = a_0$. Hence
	\begin{equation*}
	\begin{split}
	\theta(x_0) & = (\eta(x_0) \vee a_0) \wedge \mu(x_0) \\
	& = a_0 \wedge \mu(x_0) \\
	& = a_0 > \eta(x_0). 
	\end{split}
	\end{equation*} 
	Similarly, there exists $x_1 \in G$ such that $x_1 \in \mu_{a_1} \setminus \eta_{a_1}$. According to our assumption, $a_1 > a_0$. Also, since $\{\eta(x_1), a_1 \} \subseteq \text{Im}~\mu \cup \text{Im}~\eta$, by similar reasoning as above, $\eta(x_1) < a_1$. Therefore 
	\[ \theta(x_1) = (\eta(x_1) \vee a_0) \wedge \mu(x_1) = \eta(x_1) \vee a_0. \]
	Again, as $\eta$ and $\mu$ are jointly supstar, $\eta(x_1) \vee a_0 = \eta(x_1)$ or $a_0$. In either case, $\eta(x_1) \vee a_0 < a_1$. Hence
	\[ \theta(x_1) = \eta(x_1) \vee a_0 < a_1 \leq \mu(x_1), \]
	which implies that $\theta \subsetneq \mu$. Consequently, there exists $\theta \in L(\mu)$ such that $\eta \subsetneq \theta \subsetneq \mu$. But this contradicts the maximality of $\eta$. Therefore there exists exactly one $a_0 \in \text{Im}~\mu$ such that $\eta_{a_0} \subsetneq \mu_{a_0}$ and for all other $a \in \text{Im}~\mu \cup \text{Im}~\eta $, $\eta_a = \mu_a$.	   
\end{proof} 

\begin{theorem}
	\label{max_sup2}
Let $\eta \in L(\mu)$ be such that $\mu$ and $\eta$ are jointly supstar. Let $\eta$ be a maximal $L$-subgroup of $\mu$ and $\eta(e) = \mu(e)$. Then, there exists exactly one $a_0 \in \text{Im}~\mu$ such that  $\eta_{a_0}$ is a maximal subgroup of $\mu_{a_0}$ and for all other $a \in \text{Im}~\mu \cup \text{Im}~\eta $, $\eta_a = \mu_a$.
\end{theorem}
\begin{proof}
	By theorem \ref{max_sup1}, there exists exactly one $ a_0 \in \text{Im}~\mu$ such that $\eta_{a_0} \subsetneq \mu_{a_0}$ and for all other $a \in \text{Im}~\mu \cup \text{Im}~\eta$, $\eta_a = \mu_a$. Clearly, $\mu_{a_0}$ is non-empty. As $ a_0 \leq \mu(e) = \eta(e)$,  $\eta_{a_0}$ is also non-empty. Suppose, if possible, that $\eta_{a_0}$ is not a maximal $L$-subgroup of $\mu_{a_0}$. Then, there exists a subgroup $A$ of $G$ such that $\eta_{a_0} \subsetneq A \subsetneq \mu_{a_0}$.
	Define $\theta : G \rightarrow L$ as follows:
	\[ \theta(x) = \begin{cases}
	\eta(x) & \text{if } x \in \eta_{a_0} \cup (G \setminus A), \\
	a_0 & \text{if } x \in (A \setminus \eta_{a_0})
	\end{cases} \]
	for all $x \in G$. Firstly, we show that $\eta \subsetneq \theta \subsetneq \mu$. Let $x \in G$. If $x \in \eta_{a_0} \cup (G \setminus A)$, then $\theta(x) = \eta(x)$. If $x \in A \setminus \eta_{a_0}$, then $\theta(x) = a_0$. Note that $\{\eta(x), a_0 \} \subseteq \text{Im}~\eta \cup \text{Im}~\mu$. Since $\eta$ and $\mu$ are jointly supstar,  $\eta(x) \vee a_0 = \eta(x) \text{~or~} a_0$. Since, $x \notin \eta_{a_0}$, it follows that 
	\[\theta(x) = a_0 > \eta(x).\]
	\noindent Therefore $\eta \subsetneq \theta$. For $x \in G$, if $x \in \eta_{a_0} \cup (G \setminus A)$, then $\theta(x) = \eta(x) \leq \mu(x)$. If $x \in A \setminus \eta_{a_0} \subsetneq \mu_{a_0} \setminus \eta_{a_0}$, then $\theta(x) = a_0 \leq \mu(x)$. Therefore $\theta \subseteq \mu$. Now, for $x \in \mu_{a_0} \setminus A$, \[\theta(x) = \eta(x) < a_0 \leq \mu(x).\] \noindent Hence $\theta \subsetneq \mu$. Thus we have established that
	\[\eta \subsetneq \theta \subsetneq \mu.\]
	Now, we show that $\theta \in L(\mu)$. In view of Theorem \ref{lev_sgp}, it is sufficient to show  that each non-empty level subset $\theta_a$ is a subgroup of $\mu_{a}$. Hence let $\theta_{a}$ be non-empty level subset of $\mu_{a}$. We have the following cases:

	\begin{case} $a = a_0$. We show that 
	\[ \theta_{a_0} = A. \]
	\noindent Let $x \in \theta_{a_0}$. Then, $\theta(x) \geq a_0$. By definition of $\theta$, either $\theta(x) = \eta(x)$ or $\theta(x) = a_0$. This implies that 
	\[ x \in \eta_{a_0} \cup (A \setminus \eta_{a_0}) = A.\]
	\noindent Therefore $\theta_a \subseteq A$. For the reverse inclusion, let $x \in A$. Then, $x \in \eta_{a_0}$ or $x \in (A \setminus \eta_{a_0})$. In either case, $\theta(x) \geq a_0$, that is, $x \in \theta_a$. Thus $ A \subseteq \theta_a$.
	\end{case}
	\begin{case}$a > a_0$. We show that 
	\[\theta_a = \eta_a.\] 
	\noindent Since $\eta \subseteq \theta$,  $\eta_a \subseteq \theta_a$. For the reverse inclusion, let $x \in \theta_a$. Then, $\theta(x) \geq a > a_0$. By definition of $\theta$,  $\eta(x) = \theta(x) > a$, that is, $x \in \eta_a$. Thus $\theta_a = \eta_a$.
	\end{case} 
	\begin{case} $a < a_0$ and there exists no $a_1 \in \text{Im}~\eta$ such that $a \leq a_1 < a_0$. We show that \[\theta_a = A.\] 
		\noindent Let $x \in \theta_a$. Then, $\theta(x) \geq a $. This implies either $\theta(x) \geq a_0$ or $ a \leq \theta(x) < a_0$. If $\theta(x) \geq a_0$, then by the definition of $\theta$, \[x \in \eta_{a_0} \cup (A \setminus \eta_{a_0}) = A. \] 
		
		\noindent On the other hand, if $a \leq \theta(x) < a_0$, then $\theta(x) = \eta(x)$ and we have $a \leq \eta(x) < a_0$. However, this contradicts the assumption that there is no $a_1 \in \text{Im}~\eta$ such that $a \leq a_1 <a_0$. Consequently, $\theta(x) \geq a_0$ so that $x \in \theta_{a_0} $. But by Case 1,  $\theta_{a_0} = A$. This esablishes  that 
		\[\theta_a \subseteq A.\] 
		\noindent For the reverse inclusion, let $x \in A$. Then, $x \in \eta_{a_0} \cup (A \setminus \eta_{a_0})$. If $x \in \eta_{a_0}$, then $\theta(x) = \eta(x) \geq a_0 > a$. If $x \in A \setminus \eta_{a_0}$, then $\theta(x) = a_0 > a$. Thus 
		\begin{center}
		$x \in \theta_a$ for all $x \in A$.
		\end{center} 
		Hence $A \subseteq \theta_a$. 
	\end{case}
	\begin{case} $a < a_0$ and there exists $a_1 \in \text{Im}~\eta$ such that $a \leq a_1 < a_0$. We show that \[\theta_a = \eta_a.\] 
	\noindent  Since $\eta \subsetneq \theta$, $\eta_a \subseteq \theta_a$. For the reverse inclusion, let $x \in \theta_a$, that is, $\theta(x) \geq a$. Then, either $\theta(x) \geq a_0$ or $a \leq \theta(x) < a_0$. If $\theta(x) \geq a_0$, then $x \in \theta_{a_0} = A$ (in view of Case 1). Since $a_1 \in \text{Im}~\eta$ and $a_1 < a_0$,  $\eta_{a_0} \subsetneq \eta_{a_1} \subseteq \mu_{a_1}$. By theorem \ref{max_sup1}, $\eta_{a_1} = \mu_{a_1}$. Thus \[A \subsetneq \mu_{a_0} \subseteq \mu_{a_1} = \eta_{a_1} \subseteq \eta_a,\]
	\noindent and hence $x \in \eta_a$. On the other hand, if $\theta(x) < a_0$, then, by the definition of $\theta$, $\eta(x) = \theta(x) \geq a$ and hence $x \in \eta_a$. Therefore in either case, $x \in \eta_a$, so that $\theta_a \subseteq \eta_a.$ 
	\end{case}
	\begin{case} $a$ is incomparable to $a_0$. We show that \[\theta_a = \eta_a.\] 
	\noindent Since $\eta \subsetneq \theta$, hence $\eta_a \subseteq \theta_a$. For reverse inclusion, let $x \in \theta_a$. Since $a$ is incomparable with $a_0$,  $\theta(x) \neq a_0$. Hence by definition of $\theta$, \[\eta(x) = \theta(x) \geq a.\] 
	\noindent  Thus $x \in \eta_a$, so that $\theta_a  \subseteq \eta_a$.	
	\end{case}
	\noindent In all the cases, $\theta_a = \eta_a$ or $\theta_a = A$. Hence $\theta_a$ is a subgroup of $\mu_{a}$. Therefore by Theorem \ref{lev_sgp}, $\theta \in L(\mu)$. Consequently, there exists   $\theta \in L(\mu)$ such that $\eta \subsetneq \theta \subsetneq \mu$. However, this contradicts the maximality of $\eta$ in $\mu$. Hence the result. 
\end{proof}

The converse of Theorem \ref{max_sup2} does not hold. This is illustrated in the following example:

\begin{example}
	Let G be the quaternian group $Q_8$. Let $C=\{1,-1\}$  be the center of  $G$ and $H=\{\pm 1, \pm i\}$. Let the evaluation lattice $L$ be the chain given by :
	$$
	L:  0 < a < b < c < 1.
	$$
	Define $L$-subsets $\mu$ and $\eta$ of $G$ as follows:
	
	$$
	\mu (x)=\left \{ \begin{array}{ll}
	1 & {\rm if } \ x\in C,\\
	c & {\rm if } \ x\in H\setminus C,\\
	a & {\rm if} \ x\in G\setminus H;
	\end{array}
	\right.
	$$
	and
	$$
	\eta (x)=\left \{ \begin{array}{ll}
	1 & {\rm if } \ x \in C,\\
	a & {\rm if } \ x\in G\setminus C.\\
	\end{array}
	\right.
	$$

	\noindent Since the non-empty level subsets of $\eta$ and $\mu$ are subgroups of $G$,  $\eta$ and $\mu$ are $L$-subgroups of $G$. As $\eta \subseteq \mu$, $\eta$ is an $L$-subgroup of $\mu.$
	Note that $\text{Im}~\mu \cup \text{Im}~\eta = \{1,a,c\}$. Next, we observe that 
	\begin{align*}
	\eta_{a} &= ~G~ = \mu_{a},\\
	\eta_{1} &= ~C ~=\mu_{1};
	\end{align*}
	and, 	$\eta_{c} = C \subsetneq \mu_{c} = H$ is maximal. Thus  there exists exactly one $t_0 \in \text{Im}~\mu$ such that  $\eta_{t_0}$ is a maximal subgroup of $\mu_{t_0}$ and for all other $ t \in \text{Im}~\mu \cup \text{Im}~\eta $, $\eta_t = \mu_t$. However, $\eta$ is not a maximal $L$-subgroup of $\mu$. For, define an $L$-subset $\theta$ of $G$ as follows:
	$$
	\theta (x)=\left \{ \begin{array}{ll}
	1 & {\rm if } \ x\in C,\\
	b & {\rm if } \ x\in H\setminus C,\\
	a & {\rm if} \ x\in G\setminus H.
	\end{array}
	\right.
	$$
	\noindent Then, $\theta \in L(\mu)$ and $\eta \subsetneq \theta \subsetneq \mu$. 	
\end{example}

Below, we provide a sufficient condition for an $L$-subgroup to be maximal.

\begin{theorem}
	\label{max_suff}
	Let $\eta \in L(\mu)$ such that $\eta(e) = \mu(e)$ and there exists exactly one $a_0 \leq \mu(e)$ satisfying $\eta_{a_0}$ is a maximal subgroup of $\mu_{a_0}$ and for all other $a \leq \mu(e)$, $\eta_{a} = \mu_{a}$. Then, $\eta$ is a maximal $L$-subgroup of $\mu$. 
\end{theorem}
\begin{proof}
	Suppose there exists $\theta \in L(\mu)$ such that $\eta \subsetneq \theta \subseteq \mu$. Then, there exists $x_0 \in G$ such that $\eta(x_0) < \theta(x_0)$. Let $\theta(x_0) = b$. Then, $\eta_b \subsetneq \theta_b \subseteq \mu_b$. Since $\eta_a = \mu_a$ for all $a \neq a_0$, we must have $b = a_0$. Hence $\eta_{a_0} \subsetneq \theta_{a_0} \subseteq \mu_{a_0}$. By hypothesis, $\eta_{a_0}$ is a maximal subgroup of $\mu_{a_0}$. Hence $\theta_{a_0} = \mu_{a_0}$.	Thus $\theta_a = \mu_a$ for all $a \leq \mu(e)$ and we conclude that $\theta = \mu$. 
\end{proof}

The following theorem extends a well known result of classical group theory to the $L$-setting. Here we note that for $\eta \in L^{\mu}$ and $a_x \in \mu$, $\langle \eta, a_x \rangle$ is defined as the $L$-subgroup generated by $\eta \cup a_x$.

\begin{theorem}
	\label{max_lpt}
	Let $\eta\in L(\mu)$. Then, $\eta$ is maximal in $\mu$ if and only if $\langle \eta, a_x \rangle = \mu$ for all $L$-points $a_x\in\mu$ such that $a_x\notin\eta$.
\end{theorem}

\begin{proof}
	Let $\eta$ be a maximal $L$-subgroup of $\mu$ and let $a_x\in\mu$ such that $a_x\notin\eta$. Then, for all $y \in G$,
	\[ (\eta \cup a_x) (y) = \begin{cases}
	\eta(x) \vee a & \text{if }y=x,\\
	\eta(y) & \text{if }y \neq x.\\
	\end{cases} 
	\]
	Hence $\eta \subsetneq \langle \eta, a_x \rangle$. Since $\eta$ is maximal in $\mu$, we must have $\langle \eta, a_x \rangle = \mu$.
	
	\noindent Conversely, let $\langle \eta, a_x\rangle = \mu$ for all $L$-points $a_x \in \mu$ such that $a_x \notin \eta$. Let $\theta \in L(\mu)$ such that $\eta \subsetneq \theta \subseteq \mu$. Then, for some $x_0 \in G$, $\eta(x_0)<\theta(x_0)\leq\mu(x_0)$. Let $a=\theta(x_0)$. Then, $a_{x_0}\in \mu$ and $a_{x_0}\notin\eta$. By the hypothesis, $\langle \eta, a_{x_0} \rangle = \mu$. Since, $\eta \subsetneq \theta$ and $a_{x_0} \in \theta$, we get $\eta\cup a_{x_0} \subseteq \theta$. Hence $\langle \eta, a_{x_0} \rangle \subseteq \theta$.  Thus
	\[ \mu = \langle \eta, a_{x_0} \rangle \subseteq \theta \subset \mu. \] 
\end{proof}	

\begin{theorem}
	\label{max_hom}
	Let $f : G \rightarrow H$ be a group isomorphism. Let $\mu \in L(G)$ and $\nu \in L(H)$.
	\begin{enumerate}
		\item[{(i)}] If $\eta$ is a maximal $L$-subgroup of $\mu$, then $f(\eta)$ is a maximal $L$-subgroup of $f(\mu)$.
		\item[{(ii)}] If $\theta$ is a maximal $L$ subgroup of $\nu$, then $f^{-1}(\theta)$ is a maximal $L$-subgroup of $f^{-1}(\nu)$.
	\end{enumerate}
\end{theorem}
\begin{proof}
	\begin{enumerate}
	\item[{(i)}] Suppose there exists $\sigma \in L(f(\mu))$ such that $f(\eta) \subseteq \sigma \subseteq f(\mu)$.	By Proposition \ref{hom}, $\eta \subseteq f^{-1}(\sigma)$ and $f^{-1}(\sigma) \subseteq \mu$. Thus
	\[ \eta \subseteq f^{-1}(\sigma) \subseteq \mu. \]
	By Theorem \ref{hom_gp}, $f^{-1}(\sigma) \in L(\mu)$. Since $\eta$ is a maximal $L$-subgroup of $\mu$, either $f^{-1}(\sigma) = \eta$ or $f^{-1}(\sigma) = \mu$. Since $f$ is a surjection, $\sigma = f(f^{-1}(\sigma))$. Thus either $\sigma = f(\eta)$ or $\sigma = f(\mu)$. 
	\item[{(ii)}] Suppose there exists $\tau \in L(\mu)$ such that $f^{-1}(\theta) \subseteq \tau \subseteq f^{-1}(\nu)$. By Proposition \ref{hom}, $\theta \subseteq f(\tau)$ and $f(\tau) \subseteq \nu$. Thus
	\[ \theta \subseteq f(\tau) \subseteq \nu. \]
	By Theorem \ref{hom_gp}, $f(\tau) \in L(\nu)$. Since $\theta$ is a maximal $L$-subgroup of $\nu$, either $f(\tau) = \theta$ or $f(\tau) = \nu$. Since $f$ is injective, $\tau = f^{-1}(f(\tau))$. Thus either $\tau = f^{-1}(\theta)$ or $\tau = f^{-1}(\nu)$. Hence the result.
	\end{enumerate}
\end{proof}

\section{Frattini $L$-subgroup of an $L$-group}

In this section, we apply the notion of maximal $L$-subgroups to develop the notion of Frattini $L$-subgroups like their classical counterparts.  

\begin{definition}
	Let $\mu \in L(G)$. The Frattini $L$-subgroup $\Phi(\mu)$ of $\mu$ is defined to be the intersection of all maximal $L$-subgroups of $\mu$.
	
	If $\mu$ has no maximal $L$-subgroups, then we set $\Phi(\mu)=\mu$. 
\end{definition}

\begin{example}
	\label{eg1}
	Let $G=D_8$, where $D_8$ denotes the dihedral group of order 8, that is,
	\[D_8= \langle r,s~|~r^4=s^2=e,~rs=sr^{-1} \rangle. \]
	Let the evaluation lattice $L$ be the chain of five elements given by
	\[ L : 0<a<b<c<1. \] 
	\noindent Let $C=\{e, r^2\}$ be the center of $D_8$ and $K=\{e, r^{2}, s, sr^{2} \}$ be the Klein-4 subgroup of $D_8$.
	Define $\mu : G \rightarrow L$ as follows:
	\[ \mu(x) = \begin{cases}
	1 & \text{if } x=e,\\
	c & \text{if } x \in C \setminus \{e\},\\
	b & \text{if } x \in K \setminus C,\\
	a & \text{if } x \in G \setminus K.\\
	\end{cases} \]
	Since each non-empty level subset $\mu_t$ is a subgroup of $G$, by Theorem \ref{lev_gp}, $\mu \in L(G)$. We determine the Frattini $L$-subgroup of $\mu$. For this, we firstly determine all the maximal $L$-subgroups of $\mu$. Now, define the following $L$-subsets of $D_{8}$ : 
	
		\[ \eta_{1}(x) = \begin{cases}
		1 & \text{if } x=e,\\
		b & \text{if } x \in K \setminus \{e\}, \\
		a & \text{if } x \in G \setminus K;
		\end{cases}  \]
	
		\[  \eta_{2}(x) = \begin{cases}
		1 & \text{if } x=e,\\
		c & \text{if } x \in C \setminus \{e\},\\
		a & \text{if } x \in G \setminus C;\\
		\end{cases}  \]

		\[  \eta_{3}(x) = \begin{cases}
		1 & \text{if } x=e,\\
		c & \text{if } x \in C \setminus \{e\},\\
		b & \text{if } x \in K \setminus C,\\
		0 & \text{if } x \in G \setminus K.
		\end{cases}  \]

\noindent Clearly, $\eta_{i}\subseteq \mu$ for each $i$. Moreover, each non-empty level subset $(\eta_{i})_{t}$ is a subgroup of $\mu_{t}$, so by Theorem \ref{lev_sgp}, $\eta_{i} \in L(\mu)$ for each $i$. Further, observe that $\eta_{i}(e)=\mu(e)$ for each $i$ and 
\begin{align*}
(\eta_{1})_{c} \text{ is a maximal subgroup of } \mu_{c}   \text{~~and~~} (\eta_{1})_{t} &= \mu_{t}   \text{~~for all~~} t \in L \setminus \{c\},\\
(\eta_{2})_{b} \text{ is a maximal subgroup of } \mu_{b}  \text{~~and~~} (\eta_{2})_{t} &= \mu_{t}   \text{~~for all~~} t \in L \setminus \{b\},\\
(\eta_{3})_{a} \text{ is a maximal subgroup of } \mu_{a}  \text{~~and~~} (\eta_{3})_{t} &= \mu_{t}   \text{~~for all~~} t \in L \setminus \{a\}.
\end{align*}
By Theorem \ref{max_suff}, each $\eta_i$ is a maximal $L$-subgroup of $\mu$. Next, we show that $\eta_i$ are the only maximal $L$-subgroups of $\mu$ satisfying $\eta_i(e) = \mu(e)$. Suppose $\theta$ is a maximal $L$-subgroup of $\mu$ with $\theta(e) = \mu(e)$. Clearly, $\theta$ and $\mu$ are jointly supstar. Hence by theorem \ref{max_sup2}, there exists exactly one $t_0 \in \text{Im}~\mu$ such that $\theta_{t_0}$ is a maximal $L$-subgroup of $\mu_{t_0}$ and for all other $t \in \text{Im}~\theta \cup \text{Im}~\mu$, $\theta_t = \mu_t$. Note that $\text{Im}~\mu = \{a, b, c, 1\}.$ Firstly, observe $t_0\neq 1$, for if $t_0 = 1$, then $\mu_1 = \{e\}$ and $\mu_{t_0}$ has no maximal subgroups. Now, the following can be easily verified:
\begin{align*}
	\text{if } t_0 &= c, \text{ then } \theta = \eta_1;\\
	\text{if } t_0 &= b, \text{ then } \theta = \eta_2;\\
	\text{if } t_0 &= a, \text{ then } \theta = \eta_3.
\end{align*}
Consequently, $\eta_1, \eta_2$ and $\eta_3$ are the only maximal $L$-subgroups of $\mu$ such that $\eta_i(e) =\mu(e)$. 

\noindent Finally, define $\eta_4 : G \rightarrow L$ by
\[ \eta_4(x) = \begin{cases}
c & \text{if } x \in C, \\
b & \text{if } x \in K \setminus C,\\
a & \text{if } x \in G \setminus K.\\
\end{cases}  \]
Then, $\eta_4(e) \neq \mu(e)$ and by the definition of $\eta_4$ and $\mu$, it is evident that $\eta_4$ is a maximal $L$-subgroup of $\mu$. Thus we have determined all the maximal $L$-subgroups of $\mu$. Consequently, the Frattini $L$-subgroup of $\mu$ is given by:
\[  \Phi(\mu) =
\begin{cases}
c & \text{if } x=e,\\
b & \text{if } x \in C \setminus \{e\},\\
a & \text{if } x \in K \setminus C,\\
0 & \text{if } x \in G \setminus K.
\end{cases}   \] 
\end{example}

In classical group theory, the Frattini subgroup has an interesting relation to the concept of non-generators. In fact, the Frattini subgroup $\Phi(G)$ of a group $G$ turns out to be the subgroup of all  non-generators of $G$. Here, we introduce the definition of a non-generator of an $L$-group $\mu$ and establish its  above mentioned relation with the Frattini $L$-subgroup like their classical counterparts.
\begin{definition} An $L$-point $a_x\in\mu$ is said to be a non-generator of $\mu$ if, whenever $\langle \eta, a_x \rangle = \mu$ for  $\eta\in L^{\mu}$, then $\langle \eta \rangle = \mu$.
\end{definition}

Below, we prove that the set of all non-generators of $\mu$ is an $L$-subgroup of $\mu$:
\begin{theorem}
	Let $\mu \in L(G)$. Then,
	\[	\bigcup \{	a_x \mid a_x \text{ is a non generator of } \mu	\}	\]
	is an $L$-subgroup of $\mu$.
\end{theorem}

\begin{proof}
	Let $\lambda = \bigcup\{a_x \mid a_x \text{ is a non-generator of } \mu\}$. Firstly, we show \begin{center}if $a_x$ and $b_y$ are non-generators of $\mu$, then  $a_x \circ b_y $ is also a non-generator of $\mu$.\end{center}
	So, let $a_x$ and $b_y  \in \mu$ be non-generators and suppose $\eta \in L^{\mu}$ such that
	\[\langle \eta, a_x \circ b_y  \rangle = \mu. \]
	 Then, as $a_x, b_y \in \langle \eta,a_x,b_y \rangle$, we have 
	\[ \langle \eta,a_x,b_y \rangle (xy) \geq \langle \eta,a_x,b_y \rangle (x) \wedge \langle \eta,a_x,b_y \rangle (y) \geq a \wedge b.	\]
	This implies 
	\[	a_x \circ b_y =(a \wedge b)_{xy} \in \langle \eta,a_x,b_y \rangle.	\]
	Therefore
	\[ \eta \cup a_x \circ b_y \subseteq \langle \eta,a_x,b_y \rangle.	\]
	Hence it follows that 
	\[	\mu = \langle \eta,a_x \circ b_y \rangle \subseteq \langle \eta,a_x,b_y \rangle =\mu.	\]
	In view of the fact that $a_x$ and $b_y$  are non-generators of $\mu$, it follows that
	\[ \mu = \langle \eta,a_x \circ b_y \rangle = \langle \eta,a_x,b_y \rangle = \langle \eta,a_x \rangle = \langle \eta \rangle. \]
	This proves the claim. Next, to  show that $\lambda$ is an $L$-subgroup of $\mu$, consider 
	\begin{equation*}
	\begin{split}
	\lambda(xy) & = \vee \{c \mid c_{xy} \text{ is  a  non-generator of } \mu \} \\
	& \geq \vee \{ a \wedge b \mid a_x \text{ and } b_y \text{ are  non-generators of } \mu \} \\
	& \geq \{ \vee \{ a \mid a_x \text{ is a non-generator of }\mu\} \} \wedge \{ \vee \{	b \mid b_y \text{ is a non-generator of }\mu\} \}  \\
		&~~~~~~~~~~~~~~~~~~~~~~~~~~~~~~~~~~~~~~~~~~~~~~~~~~~~~~ (\text{as $L$ is a completely distributive lattice})\\
	& = \lambda(x) \wedge \lambda(y)
	\end{split}
	\end{equation*} 
	Next, we show that
	\[	\langle \eta,a_{x^{-1}} \rangle  = \langle \eta,a_x \rangle.	\]
	Note that $a_x \in \langle \eta, a_x \rangle$. Hence
	\[ \langle \eta, a_x \rangle (x^{-1}) = \langle \eta, a_x \rangle (x) \geq a. \]
	This implies $a_{x^{-1}} \in \langle \eta, a_x \rangle$ so that $\langle \eta, a_{x^{-1}} \rangle \subseteq \langle \eta, a_x \rangle$.
	Similarly, $\langle \eta, a_x \rangle \subseteq  \langle \eta, a_{x^{-1}} \rangle$. Thus
	\[ \langle \eta, a_{x^{-1}} \rangle = \langle \eta, a_x \rangle. \]
	Hence $ \lambda(x^{-1}) = \lambda(x)$ for all $x \in G$. Consequently, $\lambda \in L(\mu)$.
\end{proof}

Recall that a chain is said to be upper well ordered if  every non-empty subset of the given chain has a supremum. Clearly, every subset of an upper well ordered chain is a supstar subset. Consequently, by Proposition \ref{supchar},
each $L$-subset $\eta$ of an upper well ordered chain $L$ satisfies sup-property. In fact, we have the following :

\label{new}
\begin{proposition}
	Let $L$ be an upper well ordered chain. Then, any family $\{\eta_{i}\}_{i \in I} \subseteq L^\mu$ is a supstar family.  
	
\end{proposition}
\begin{lemma}
	\label{zrn}
	Let $L$ be an upper well-ordered chain and let $\mu$ be an $L$-subgroup of $G$. Suppose that $\theta \in L(\mu)$ and let $a_x$ be an $L$-point of $\mu$ such that $a_x \notin \theta$. Then, there exists $\eta \in L(\mu)$ such that $\eta$ is maximal with respect to the conditions $\theta \subseteq \eta$ and $a_x \notin \eta$.
\end{lemma}

\begin{proof}
	Consider the set $S = \{ \nu \in L(\mu) \mid \theta \subseteq \nu \text{ and } a_x \notin \nu \}$. Then, $S$ is non-empty since $\theta \in S$. Also, $S$ is partially ordered by the $L$-set inclusion  $\subseteq$. Let $C = \{ \theta_i \}_{i \in I}$ be a chain in $S$. Then, we claim that 
	\[\bigcup_{i \in I} \theta_i \in S. \]  
	Firstly, we show that $\bigcup_{i \in I} \theta_i \in L(\mu)$. For this, let $x, y \in G$ and consider
	\begin{align*}
	\bigcup_{ {i \in I}} \theta_{i}~(xy) &= \bigvee_{i\in I} \theta_{i}(xy)\\
	&\geq \bigvee_{i\in I} \{\theta_{i}(x) \wedge \theta_{i}(y)\}\\
	&= \bigvee_{i\in I} \theta_{i}(x) \text{~or~} \bigvee_{i\in I} \theta_{i}(y) \end{align*}~~~~~~~~~~~~~~~~~~~~~~~~~~~~~~~~~~~~~~~~~~~~~~~~~~~~~~~~~~~~~~~~~~ \text{(as $L$ is a chain, $\theta_{i}(x) \wedge \theta_{i}(y) = \theta_{i}(x) $  or $\theta_{i}(y) $)}
	\begin{align*}
	~~~~~~~~~~~~~~~~~~~~~~~	&\geq \left\{ \bigvee_{i\in I} \{\theta_{i}(x)\} \right\} \bigwedge \left\{\bigvee_{i\in I} \{\theta_{i}(y)\}\right\}\\
	&= \bigcup_{{i \in I}}\theta_i(x) \bigwedge \bigcup_{{i \in I}}\theta_i(y).
	\end{align*}	
	As $\theta_{i} \in L(\mu)$, it follows that 
	$$\bigcup_{ {i \in I}} \theta_{i}~(x^{-1})=\bigcup_{ {i \in I}} \theta_{i}~(x).$$
	\noindent Now, it is clear that $\theta \subseteq \bigcup_{i \in I} \theta_i$. Also, since $L$ is upper well-ordered and $a_x \notin \theta_i$ for all $i \in I$, $a_x \notin \bigcup_{i \in I} \theta_i$. Hence $\bigcup_{i \in I} \theta_i \in S$ so that every chain in  $S$ has an upper bound. Therefore by Zorn's lemma, $S$ has a maximal element $\eta$. This proves the result.
\end{proof}

\begin{theorem}
	\label{fra_lpt}
	Let $\mu \in L(G)$ and let $\lambda$ be the $L$-subgroup of non-generators of $\mu$. Then,
	\[	\lambda \subseteq \Phi(\mu). \]
	The equality holds if $L$ is an upper well ordered chain.
\end{theorem}
\begin{proof}
	Let $a_x \in \mu$ be a non-generator of $\mu$ and let $\eta$ be a maximal subgroup of $\mu$. Suppose, if possible, that $a_x \notin \eta$. Then, by Theorem \ref{max_lpt}, $\langle \eta,a_x \rangle =\mu$. However, since $a_x$ is a non-generator of $\mu$, we get
	\[	\mu = \langle \eta,a_x \rangle = \langle \eta \rangle = \eta = \mu,	\]
	which contradicts the assumption that $\eta$ is a maximal $L$-subgroup of $\mu$. Thus $a_x \in \eta$ for all maximal subroups of $\mu$. It follows that 
	\[ \lambda(x) = \bigcup\{a_x \mid a_x \text{ is a non-generator of } \mu\} \subseteq \Phi(\mu).	\]
	Next, suppose that $L$ is an upper well ordered chain and let $x \in G$. Let $b = \Phi(\mu)(x)$. Then, $b_x \in \Phi(\mu)$. We show that $b_x$ is a non-generator of $\mu$. Suppose, if possible, that there exists $\eta \in L^{\mu}$ such that $\mu = \langle \eta, b_x \rangle$ and $\mu \neq \langle \eta \rangle$. Then, $b_x \notin \langle \eta \rangle$ and hence by Lemma \ref{zrn}, there exists an $L$-subgroup $\theta$ of $\mu$ which is maximal subject to the conditions $\langle \eta \rangle \subseteq \theta$ and $b_x \notin \theta$. We show that $\theta$ is a maximal $L$-subgroup of $\mu$. 
	
	\noindent If $\theta \subsetneq \nu \subseteq \mu$ for some $\nu \in L(\mu)$, then $\eta \subseteq \langle \eta \rangle \subseteq \nu$. Since $\theta$ is maximal with respect to the conditions that $\langle \eta \rangle \subseteq \theta$ and $b_x \notin \theta$, we must have $b_x \in \nu$. This implies that $\mu = \langle \eta, b_x \rangle \subseteq \nu$ and hence $ \nu = \mu$. Consequently, $\theta$ is a maximal $L$-subgroup of $\mu$. But by the maximality of $\theta$, it follows that
	\[ b_x \in \Phi(\mu) = \bigcap \{ \eta_i \mid \eta_i \text{ is a maximal $L$ -subgroup of } \mu \} \subseteq \theta, \]
	contradicting the assumption that $b_x \notin \theta$. Hence $b_x$ is a non-generator of $\mu$. Therefore
	\[	\lambda(x) = \bigvee\{a \mid a_x \text{ is a non-generator of } \mu\} \geq b = \Phi(\mu)(x). 	\]	 
\end{proof}

In the following example, we construct the Frattini $L$-subgroup $\Phi(\mu)$ of an $L$-group $\mu$ by using the concept of non-generators:
 
\begin{example}
	\label{eg2}
	Consider $\mu \in L(G)$ as given in Example \ref{eg1}. Firstly, we note that $L$ being a finite chain is upper well-ordered. We determine the $L$-subgroup of non-generators $\lambda$ of $\mu$.
	\noindent We show that $\lambda(r^2) = b$.
	For this, we claim that $b_{r^2}$ is a non-generator of $\mu$. Let $\theta \in L^{\mu}$ such that $\langle \theta, b_{r^2} \rangle = \mu$. Since $L$ is a finite chain, $\theta \cup b_{r^2}$ possesses the sup-property. Hence by Theorem \ref{gen_sup}, 
	\[ \langle (\theta \cup b_{r^2})_c \rangle = \langle \theta, b_{r^2} \rangle_c = \langle \mu \rangle_c = \mu_c = \{ e, r^2 \} = \langle r^2 \rangle.  \]
	This implies that $r^2 \in (\theta \cup b_{r^2})_c$. Now, 
	\[ (\theta \cup b_{r^2})(r^2) = \theta(r^2) \vee b \geq c > b. \]
	Hence we must have $\theta(r^2) \geq c$, that is, $c_{r^2} \in \theta$, which implies that $\theta \cup b_{r^2} = \theta$. Thus $\langle \theta \rangle = \mu$. We conclude that $b_{r^2}$ is a non-generator of $\mu$. Hence 
	\begin{equation} 
	\lambda(r^2) \geq b.
	\end{equation} 
	\noindent Next, we show that $c_{r^2}$ is not a non-generator of $\mu$. It can be easily seen that $\theta : G \rightarrow L$ given by
	\[ \theta(x) = \begin{cases}
	1 & \text{if } x=e,\\
	b & \text{if } x \in K \setminus \{e\},\\
	a & \text{if } x \in G \setminus K\\
	\end{cases} \]
	is an $L$-subgroup of $\mu$ such that $\langle \theta,c_{r^2} \rangle = \mu$ but $\langle \theta \rangle = \theta \neq \mu$. Hence $c_{r^2}$ is not a non-generator of $\mu$. Thus
	\begin{equation}
	\lambda(r^2) < c.
	\end{equation}
	From (1) and (2), we conclude that $\lambda(r^2) = b$. By similar calculations, all the non-generators of $\mu$ can be determined, and we get
	\begin{equation*}
	\begin{split}
	\lambda(\mu) &= \begin{cases}
	c & \text{if } x=e, \\
	b & \text{if } x \in C \setminus \{e\},\\
	a & \text{if } x \in K \setminus C, \\
	0 & \text{if } x \in G \setminus K
	\end{cases}\\
	&= \Phi(\mu),
	\end{split}
	\end{equation*}
	as determined in Example \ref{eg1}.  
\end{example}

\begin{proposition}
	Let $L$ be an upper well ordered chain and let $\mu \in L(G)$ such that $\Phi(\mu)(e) = \mu(e)$. Then, $\Phi(\mu_b) \subseteq (\Phi(\mu))_b$ for all $b \in \text{Im}~\mu$.
\end{proposition}
\begin{proof}
	Let $b \in \text{Im}~\mu$ and let $\eta$ be a maximal $L$-subgroup of $\mu$. Since $L$ is upper well ordered, $\eta$ and $\mu$ are jointly supstar. Moreover, $\Phi(\mu)(e) = \mu(e)$ implies that $\eta(e) = \mu(e)$. By theorem \ref{max_sup2}, there exists exactly one $a_0 \in \text{Im}~\mu$ such that $\eta_{a_0}$ is a maximal subgroup of $\mu_{a_0}$ and for all other $a \in \text{Im}~\mu \cup \text{Im}~\eta$, $\eta_a = \mu_a$. Hence either $\eta_b$ is a maximal subgroup of $\mu_b$ or $\eta_b = \mu_b$. In both the cases, $\Phi(\mu_b) \subseteq \eta_b$. Since $\eta$ is any arbitrary maximal $L$-subgroup of $\mu$, we get 
	\[ \begin{split}
	\Phi(\mu_b) & \subseteq \bigcap \{ \eta_b \mid \eta \text{ is a maximal $L$-subgroup of } \mu \}\\
	&= \left( \bigcap \{ \eta \mid \eta \text{ is a maximal $L$-subgroup of } \mu \} \right)_b ~~~~~ \text{(by Proposition \ref{int_lev})}\\
	&= (\Phi(\mu))_b.
	\end{split} \]
\end{proof}

\begin{example}
	Consider the $L$-group $\mu$ of example \ref{eg1}. Then, $\Phi(\mu)$ is given by 
	\[ \Phi(\mu) =
	\begin{cases}
	c & \text{if } x=e,\\
	b & \text{if } x \in C \setminus \{e\},\\
	a & \text{if } x \in K \setminus C,\\
	0 & \text{if } x \in G \setminus K.
	\end{cases}   \]
	Hence we see that $(\Phi(\mu))_b = \{ e, r^2 \}$. However, $\mu_b = K$ and thus $\Phi(\mu_b) = \{ e \}$. This shows that $(\Phi(\mu))_b \nsubseteq \Phi(\mu_b)$.  
\end{example}

\begin{lemma}
	\label{lpt_con}
	Let $\mu$ be a normal $L$-subgroup of $G$. If $a_x$ is a non-generator of $\mu$, then $a_{gxg^{-1}}$ is a non-generator of $\mu$ for all $g \in G$.
\end{lemma}
\begin{proof}
	Let $a_x$ be a non-generator of $\mu$ and let $g \in G$. Suppose, if possible, that $a_{gxg^{-1}}$ is not a non-generator of $\mu$. Then, there exists an $L$-subset $\eta$ of $\mu$ such that $\langle \eta, a_{gxg^{-1}} \rangle = \mu$ but $\langle \eta \rangle \neq \mu$.
	
	\noindent Define $\theta : G \rightarrow L$ as follows:
	\[ \theta(z) = \eta(gzg^{-1}) ~~~~~ \text{for all $z \in G$.} \]
	Then, $\theta \in L^{\mu}$, for if $\theta(z) > \mu(z)$ for some $z \in G$, then $\eta(gzg^{-1}) > \mu(z) = \mu(gzg^{-1})$, which contradicts $\eta \in L^{\mu}$. Also,
	\[ \text{tip }\theta = \bigvee_{z \in G}\theta(z) = \bigvee_{z \in G}\eta(gzg^{-1}) = \text{tip }\eta. \]
	We claim that whenever $y \in \langle (\eta \cup a_{gxg^{-1}})_c \rangle$ for some $c \in L$, then $g^{-1}yg \in \langle (\theta \cup a_x)_c \rangle$.
	
	\noindent Let $c \in L$ and let $y \in \langle (\eta \cup a_{gxg^{-1}})_c \rangle$. Then, 
	\[ y = y_1y_2\ldots y_n, \text{ where } y_i \text{ or } {y_i}^{-1} \in (\eta \cup a_{gxg^{-1}})_c. \]
	Then,
	\[ g^{-1}yg = (g^{-1}y_1g)(g^{-1}y_2g)\ldots(g^{-1}y_ng). \]
	Note that for all $1 \leq i \leq n$,
	\[ (\eta \cup a_{gxg^{-1}})(y_i) = \begin{cases}
	\eta(y_i) & \text{if } y_i \neq gxg^{-1},\\
	\eta(y_i) \vee a & \text{if } y_i = gxg^{-1}
	\end{cases} \]
	and in view of the definition of $\theta$, we obtain
	\begin{align*}
	(\theta \cup a_x)(g^{-1}y_ig) &= \begin{cases}
	\theta(g^{-1}y_ig) & \text{if }g^{-1}y_ig \neq x, \\
	\theta(g^{-1}y_ig) \vee a & \text{if } g^{-1}y_ig =x
	\end{cases} \\
	&= \begin{cases}
	\eta(y_i) & \text{if } y_i \neq gxg^{-1}, \\
	\eta(y_i) \vee a & \text{if } y_i = gxg^{-1}.
	\end{cases} \\
	\end{align*}
	Hence 
	\[ (\eta \cup a_{gxg^{-1}})(y_i) = (\theta \cup a_x)(g^{-1}y_ig) ~~~~~ \text{ for all } 1 \leq i \leq n. \]
	This implies that $y \in \langle (\eta \cup a_{gxg^{-1}})_c \rangle$ if and only if $g^{-1}yg \in \langle (\theta \cup a_x)_c \rangle$ for all $c \in L$. By theorem \ref{gen}, for $z \in G$, 
	\[ \langle \eta, a_{gxg^{-1}} \rangle (z)= \bigvee_{c \leq a_0} \{ c \mid z \in \langle (\eta \cup a_{gxg^{-1}})_c \rangle  \}, \]
	where $a_0 = \text{tip}(\eta \cup a_{gxg^{-1}}) = \text{tip}(\theta \cup a_x)$.
	Therefore $\langle \eta, a_{gxg^{-1}} \rangle (y) = \langle \theta, a_x \rangle (g^{-1}yg).$
	Similarly, $u \in \langle \eta_c \rangle $ for some $c \in L$ if and only if
	\[ u = u_1u_2 \ldots u_m \text{ where } u_i \text{ or } {u_i}^{-1} \in \eta_c, \]
	if and only if 
	\[ g^{-1}ug = (g^{-1}u_1g)(g^{-1}u_2g) \ldots (g^{-1}u_mg). \]
	By using similar arguments as above, it can be verified that $u_i \in \eta_c$ if and only if $g^{-1}u_ig \in \theta_c$, and we get $y \in \langle \eta_c \rangle$ if and only if $g^{-1}yg \in \langle \theta_c \rangle$.
	Thus $\langle \eta \rangle (y) = \langle \theta \rangle (g^{-1}yg)$.  
	
	\noindent Now, we show that $ \langle \theta, a_x \rangle = \mu.$ Let $y \in G$ and let $b = \mu(y)$. Since $\mu$ is a normal $L$-subgroup of $G$, $\mu(gyg^{-1}) = \mu(y) = b$, which gives $b_{gyg^{-1}} \in \mu$. Since $ \langle \eta, a_{gxg^{-1}} \rangle = \mu$, $b_{gyg^{-1}} \in \langle \eta, a_{gxg^{-1}} \rangle$. Thus $\langle \eta, a_{gxg^{-1}} \rangle (gyg^{-1}) \geq b$, which implies that $\langle \theta, a_x \rangle(y) \geq b = \mu(y)$. Hence we conclude that $\langle \theta, a_x \rangle = \mu$.
	
	\noindent Next, since $ \langle \eta \rangle \neq \mu$, there exists an $y \in G$ such that $\mu(y) > \langle \eta \rangle (y)$. Then, $\langle \theta \rangle (g^{-1}yg) = \langle \eta \rangle (y)$, which implies that $\langle \theta \rangle (g^{-1}yg) < \mu(y) = \mu(g^{-1}yg)$. Thus $\langle \theta \rangle \neq \mu$. Hence there exists $\theta \in L^{\mu}$ such that $\langle \theta, a_x \rangle = \mu$ but $\langle \theta \rangle \neq \mu$. This contradicts the assumption that $a_x$ is a non-generator of $\mu$. Hence the result.
\end{proof}

\begin{theorem}
	\label{fra_nor}
	Let $L$ be an upper well ordered chain and let $\mu$ be a normal $L$-subgroup of $G$. Then, $\Phi(\mu)$ is a normal $L$-subgroup of $\mu$.
\end{theorem}
\begin{proof}
	Let $x, g \in G$. Then, $(\Phi(\mu)(x))_x \in \Phi(\mu)$. By Theorem \ref{fra_lpt}, $(\Phi(\mu)(x))_x$ is a non-generator of $\mu$. By Lemma \ref{lpt_con}, $(\Phi(\mu)(x))_{gxg^{-1}}$ is a non-generator of $\mu$. Hence $(\Phi(\mu)(x))_{gxg^{-1}} \in \Phi(\mu)$. Therefore we get
	\[ \Phi(\mu)(gxg^{-1}) \geq \Phi(\mu)(x) \geq \Phi(\mu)(x) \wedge \mu(g). \]
	Thus we conclude that $\Phi(\mu)$ is a normal $L$-subgroup of $\mu$.
\end{proof}

The following example illustrates the above theorem:

\begin{example}
	Consider the $L$-group $\mu$ given in Example \ref{eg1}. Then, since $L$ is a finite chain, it is upper well ordered. Also, since every non-empty level subset of $\mu$ is a normal subgroup of $G$, by Theorem \ref{lev_norgp}, $\mu$ is a normal $L$-subgroup of $G$. In the Example \ref{eg1}, $\Phi(\mu)$ is defined to be
	\[ \Phi(\mu) = \begin{cases}
	c & \text{if } x=e,\\
	b & \text{if } x \in C \setminus \{e\},\\
	a & \text{if } x \in K \setminus C,\\
	0 & \text{if } x \in G \setminus K.
	\end{cases} \]
	Note that $(\Phi(\mu))_t$ is a normal subgroup of $\mu_t$ for all non-empty level subsets. Hence by Theorem \ref{lev_norsgp}, $\Phi(\mu)$ is a normal $L$-subgroup of $\mu$. 
\end{example}

\begin{theorem}
	Let $f : G \rightarrow H$ be a group isomorphism and let $\mu \in L(G)$. Then,
	\[ f(\Phi(\mu)) \subseteq \Phi(f(\mu)). \]
\end{theorem}
\begin{proof}
	By Theorem \ref{max_hom}, $\eta$ is a maximal $L$-subgroup of $\mu$ if and only $f(\eta)$ is a maximal $L$-subgroup of $f(\mu)$. Thus
	\begin{align*}
		\begin{split}
		f(\Phi(\mu)) &= f \left( \bigcap \left\{ \eta_i \mid \eta_i \text{ is a maximal $L$-subgroup of } \mu \right\} \right) \\
		& \subseteq \bigcap \left\{ f(\eta_i) \mid \eta_i \text{ is a maximal $L$-subgroup of } \mu \right\} ~~~~~ \text{(By Proposition \ref{hom})} \\ 
		&= \Phi(f(\mu)).
		\end{split}
	\end{align*}
\end{proof}

\section{Conclusion}

After the concept of fuzzy subgroups was introduced by Rosenfeld, so far the researchers have studied the fuzzy subgroups and fuzzy subrings of a classical group and a classical ring, respectively. In our studies, we have shifted to the $L$-(fuzzy) subgroups where the parent structure is a $L$-(fuzzy) group instead of an ordinary group. This has resulted in the examination of various concepts such as nilpotent $L$-subgroup of an $L$-group, solvable $L$-subgroup of an $L$-group, normalizer of an $L$-group, etc. This paper carries forward this approach further by defining the concepts of maximal $L$-subgroups of an $L$-group, Frattini $L$-subgroup of an $L$-group and non-generators of an $L$-group.

The research in the discipline of fuzzy algebraic structures came to a standstill after  Tom Head's metatheorem and subdirect product theorem. This is because most of the concepts and results in the studies of fuzzy algebra could be established through simple applications of the metatheorem and the subdirect product theorem. However, the metatheorem and the subdirect product theorems are not applicable in the $L$-setting. Hence we suggest the researchers pursuing studies in these areas to investigate the properties of $L$-subalgebras of an $L$-algebra rather than $L$-subalgebras of classical algebra. 

As an application and motivation here we mention that
if we replace the lattice $L$ in our work by the closed unit interval $[0,1]$, then we retrieve the corresponding version of fuzzy group theory. Moreover, as an application of this theory
we also mention that if we replace the lattice $L$ by the two elements set $\{0, 1\}$, then the results of classical group theory follow as simple corollaries of the corresponding results of $L$-group theory. This way, the $L$-group theory provides
us a new language and a new tool for the study of the classical group theory. The classical group theory has been founded on abstract sets and therefore the language used for its development
is formal set theory. On the other hand, $L$-group theory expresses itself  through the language of (lattice valued) functions. This shift of study from the language of sets to the language of functions gives rise to new insights that are the main focus of our work.

\section*{Acknowledgements}
The second author of this paper was supported by the Junior Research Fellowship joinly funded by CSIR and UGC, India during the course of development of this paper.

\end{document}